\documentclass[reqno,10pt]{amsart} 

\usepackage{comment, amsmath, amssymb, amsgen, amsthm, amscd, xspace, color, epsfig, float, epic}
\usepackage{genyoungtabtikz}
\usepackage{pifont}
\usepackage{multirow}
\usepackage{threeparttable}
\usepackage[colorlinks,linkcolor=blue,urlcolor=cyan,citecolor=blue,pagebackref]{hyperref}

\newcommand{\sob}{\mathfrak{so}(2n+1,\mathbb{C})}

\newcommand{\sod}{\mathfrak{so}(2n,\mathbb{C})}

\newcommand{\hobox}[3]{\draw (0+#1,0-#2) rectangle (1+#1,-1-#2)++(-0.5,+0.5) node {$ #3$};}

\newcommand{\domscale}{0.5}
\newcommand{\gl}{{\mathfrak{gl}}}
\newcommand{\gln}{{\gl_n}}

\newcommand{\spg}{{\mathfrak{sp}}}
\newcommand{\spn}{{\spg_{2n}}}

\newcommand{\GK}{\mathrm{GKdim}}

\newcommand{\pf}{\begin{proof}}
	\newcommand{\epf}{\end{proof}}
\newcommand{\eq}{\begin{equation}}
	\newcommand{\eeq}{\end{equation}}
\newcommand{\eqn}{\begin{equation*}}
	\newcommand{\eeqn}{\end{equation*}}

\newcommand{\frg}{\mathfrak{g}}
\newcommand{\frh}{\mathfrak{h}}

\newcommand{\frl}{\mathfrak{l}}

\newcommand{\frn}{\mathfrak{n}}

\newcommand{\frq}{\mathfrak{q}}

\newcommand{\fru}{\mathfrak{u}}

\theoremstyle{definition}
\newtheorem{Thm}[equation]{Theorem}

\newtheorem{prop}[equation]{Proposition}
\newtheorem{lem}[equation]{Lemma}

\newtheorem{Rem}[equation]{Remark}

\newtheorem{definition}[equation]{Definition}

\newtheorem{example}[equation]{Example}

\usepackage{enumitem}
\setlist[enumerate,1]{font=\textup,
	leftmargin=0mm,labelsep=0.2em,topsep=1mm,itemsep=1mm,itemindent=1em,listparindent=1em}
\setlist[enumerate,2]{font=\textup, leftmargin=3mm,labelsep=1mm,topsep=1mm,itemsep=1mm,itemindent=1em}
\setlist[enumerate,3]{font=\textup, leftmargin=4mm,labelsep=1mm,topsep=1mm,itemsep=1mm,itemindent=1em,listparindent=1em}
\setlist[enumerate,4]{font=\textup, leftmargin=5mm,labelsep=1mm,topsep=1mm,itemsep=1mm,itemindent=1em,listparindent=1em}
\numberwithin{equation}{section}
\begin{document}

\title[Reducibility of classical Lie algebras]{Reducibility of scalar  generalized Verma modules of minimal parabolic type}

\author{Jing Jiang}

\address[Jiang]{School of Mathematical Sciences, East China Normal University, Shanghai 200241, P. R. China}
\email{jjsd6514@163.com}

\thanks{Supported by the National Science Foundation of China (No. 12171344)}



\begin{abstract}
Let $\mathfrak{g}$ be a classical complex simple Lie algebra and $\mathfrak{q}$ be a parabolic subalgebra. Generalized Verma module $M$ is called a scalar generalized Verma module if it is induced from a one-dimensional representation of $\mathfrak{q}$. In this paper, we will determine the first diagonal-reducible point of scalar generalized Verma modules associated to minimal parabolic subalgebras by computing explicitly the Gelfand-Kirillov dimension of the corresponding highest weight modules.

\noindent{\textbf{2020 Mathematics Subject Classification:} 16S30, 17B10, 17B20, 22E47
}

\noindent{\textbf{Keywords:} Generalized Verma module; Gelfand-Kirillov dimension; Minimal parabolic subalgebra}
\end{abstract}

\maketitle

 \tableofcontents

\section{ Introduction}

	Let $\mathfrak{g}$ be a finite-dimensional complex simple Lie
algebra and $U(\frg)$ be its universal enveloping algebra.
Fix a Cartan subalgebra $\mathfrak{h}$ and denote by $\Delta$ the root system associated to $(\frg, \frh)$. Choose a positive root system
$\Delta^+\subset\Delta$ and a simple system $\Pi\subset\Delta^+$. Let $\rho$ be the half sum of roots in $\Delta^+$. Denote the triangular decomposition of $\mathfrak{g}$ with nilpotent radical $\mathfrak{n}=\mathop\oplus\limits_{\alpha\in\Delta^+}\frg_\alpha$ and its opposite nilradical $\bar{\mathfrak{n}}$ as $\mathfrak{g}=\bar{\mathfrak{n}}\oplus\mathfrak{h}\oplus{\mathfrak{n}}$.  Choose a subset $I\subset\Pi$ and  it generates a subsystem
$\Delta_I\subset\Delta$.
Let $\frq_I$ be the standard parabolic subalgebra corresponding to $I$ with Levi decomposition $\frq_I=\frl_I\oplus\fru_I$. When $I=\emptyset$, we have $\frq_{\emptyset}=\mathfrak{h}\oplus
\mathfrak{n}=\mathfrak{b}$. 

Let $\frq_I=\frl_I\oplus\fru_I$ and $\mathbb{C}(\lambda)$ be a finite-dimensional irreducible $\mathfrak{l}_I$-module with the highest weight $\lambda\in\frh^*$. The {\it generalized Verma module} $M_I(\lambda)$ is defined by
\[
M_I(\lambda):=U(\frg)\otimes_{U(\frq)}\mathbb{C}(\lambda).
\]
In particular, $M(\lambda)=M_{\emptyset}(\lambda)$ is called a {\it Verma module}.
The irreducible quotient of $M(\lambda)$ is denoted by $L(\lambda)$. It is also the irreducible quotient of $M_I(\lambda)$. In the case when $\dim(\mathbb{C}(\lambda))=1$, $M_I(\lambda)$ is called a {\it scalar generalized Verma module}.

The theory of highest weight modules over simple complex finite-dimensional Lie algebras rests on the original work of Verma, as presented in his seminal paper \cite{VD}. In that work, Verma introduced and studied a family of universal highest weight modules known as Verma modules, which has become fundamental to the field.

Several attempts have been made to extend the theory of Verma modules, and one of the most natural ways to achieve this is by generalizing Verma modules themselves. This can be accomplished in various ways, such as in \cite{AG1,AR,JL}. Generalized Verma modules (GVM) have been investigated from different points, and many of the properties of classical Verma modules have been either proven for GVM or generalized to them. For the study of GVM, see Mazorchuk’s work in \cite{MV}. In this article, he provided a comprehensive study of GVM using parabolic induction for a parabolic subalgebra of a simple Lie algebra, as well as gave an overview of the subject’s historical development.

The reducibility problem for generalized Verma modules is of great importance in representation theory and has close connections to several other problems as documented in \cite{BXiao, EHW, MH}. The crucial tool for solving this problem is Jantzen's criterion \cite{JC}, although it can be quite complicated to apply in practice. However, Kubo \cite{Ku} discovered some practical reducibility criteria for solving this problem for scalar generalized Verma modules associated with exceptional simple Lie algebras and certain maximal parabolic subalgebras. Using Kubo's result, He \cite{He} established reducibility for all scalar generalized Verma modules of Hermitian symmetric pairs. Then He-Kubo-Zierau \cite{HKZ} extended this to all scalar generalized Verma modules associated with maximal parabolic subalgebras. Recently, Bai-Xiao \cite{BXiao} resolved the reducibility problem for all generalized Verma modules of Hermitian symmetric pairs.

Gelfand-Kirillov dimension
plays a crucial role in characterizing algebraic structures with infinite dimensions. It has been used since Joseph's work in \cite{AJ1} to measure the size of Lie algebras and Lie group representations. A recent endeavor led by Bai-Xiao demonstrated that if the Gelfand-Kirillov dimension of its simple quotient of a scalar generalized Verma module is smaller than the dimension of $\fru$, then that module is reducible. ( Our approach does not rely on the simplification methods outlined in \cite{HKZ}.) We can utilize the same technique employed in \cite{BXX} to compute the GK dimension of scalar type highest weight modules.

	The paper is organized as follows. The necessary preliminaries for minimal parabolic subalgebra and Gelfand-Kirillov dimension are given in Section $2$. In Section $3$, we give the reducibility of scalar generalized Verma modules for type $A_n$, $B_n$, $C_n$ and $D_n$.

\section{Preliminaries}
In this section, we will give brief preliminaries on GK dimension, Young tableau and parabolic subalgebra. See \cite{VB} and \cite{Vo78} for more details.

\begin{definition}
	A parabolic subalgebra $\mathfrak{q}$ is said to be minimal if $I$ is the minimal nonempty set of $\Pi$. In other words, there is only one element in $I$.	
\end{definition}

Let $M$ be a finite generated $U(\mathfrak{g})$-module. Fix a finite dimensional generating subspace $M_0$ of $M$. Let $U_{n}(\mathfrak{g})$ be the standard filtration of $U(\mathfrak{g})$. Set $M_n=U_n(\mathfrak{g})\cdot M_0$ and
\(
\text{gr} (M)=\bigoplus\limits_{n=0}^{\infty} \text{gr}_n M,
\)
where $\text{gr}_n M=M_n/{M_{n-1}}$. Thus $\text{gr}(M)$ is a graded module of $\text{gr}(U(\mathfrak{g}))\simeq S(\mathfrak{g})$.

\begin{definition} The \textit{Gelfand-Kirillov dimension} of $M$  is defined by
	\begin{equation*}
		\operatorname{GKdim} M = \varlimsup\limits_{n\rightarrow \infty}\frac{\log\dim( U_n(\mathfrak{g})M_{0} )}{\log n}.
	\end{equation*}
\end{definition}

It is easy to see that the above definition is independent of the choice of $M_0$.

Then we have the following lemma.
\begin{lem}[{\cite[Lemma 4.4]{BX}}]\label{GKdown}
	For any $z\in\mathbb{C}$, we have
	\[
	\GK(L((z+1)\omega))\leq\GK(L(z\omega)).
	\]
	In particular, if $M_I(z\omega)$ is reducible, then $M_I((z+1)\omega)$ is also reducible.
\end{lem}

The following lemma is very useful in our proof of the main results.
\begin{lem}[{\cite[Theorem 1.1]{BX}}]\label{reducible}
	A scalar generalized Verma module $M_I(\lambda)$ is irreducible if and only if ${\rm GKdim}\:L(\lambda)=\dim(\mathfrak{u})$.
\end{lem}

In \cite{BXX} and  \cite{BX1}, the authors  have found an algorithm to compute the  Gelfand-Kirillov dimension of highest weight modules of classical Lie algebras. We recall their algorithms here.

For  a totally ordered set $ \Gamma $, we  denote by $ \mathrm{Seq}_n (\Gamma)$ the set of sequences $ x=(x_1,x_2,\cdots, x_n) $   of length $ n $ with $ x_i\in\Gamma $.
We say $q=(q_1, \cdots, q_N)$ is the \textit{dual partition} of a partition $p=(p_1, \cdots, p_N)$ and write $q=p^t$ if $q_i$ is the length of $i$-th column of the Young diagram $p$. 
Let $p(x)$ be the shape of the Young tableau $P(x)$ obtained by applying Robinson-Schensted algorithm (\cite{BX1,VB}) to $x\in \mathrm{Seq}_n (\Gamma)$. For convenience, we set $q(x)=p(x)^t$.

\begin{example}Let $x=(5,4,1,3,2,6)$. Then by using RS-insertion algorithm, we have
	\[
	\small
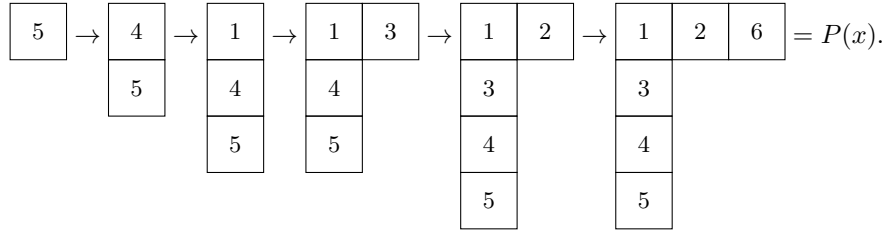
{\begin{tikzpicture}[scale=\domscale+0.25,baseline=-15pt]
			\hobox{0}{0}{5}
	\end{tikzpicture}}\to
	\small{\begin{tikzpicture}[scale=\domscale+0.25,baseline=-15pt]
			\hobox{0}{0}{4}
			\hobox{0}{1}{5}
	\end{tikzpicture}}\to
	\small{\begin{tikzpicture}[scale=\domscale+0.25,baseline=-15pt]
			\hobox{0}{0}{1}
			\hobox{0}{1}{4}
			\hobox{0}{2}{5}
	\end{tikzpicture}}	\to
	\small{\begin{tikzpicture}[scale=\domscale+0.25,baseline=-15pt]
			\hobox{0}{0}{1}
			\hobox{1}{0}{3}
			\hobox{0}{1}{4}
			\hobox{0}{2}{5}
	\end{tikzpicture}}\to
    	\small{\begin{tikzpicture}[scale=\domscale+0.25,baseline=-15pt]
    		\hobox{0}{0}{1}
    		\hobox{1}{0}{2}
    		\hobox{0}{1}{3}
    		\hobox{0}{2}{4}
    		\hobox{0}{3}{5}
    \end{tikzpicture}}\to  
	\small{\begin{tikzpicture}[scale=\domscale+0.25,baseline=-15pt]
			\hobox{0}{0}{1}
			\hobox{1}{0}{2}
			\hobox{2}{0}{6}
			\hobox{0}{1}{3}
			\hobox{0}{2}{4}
			\hobox{0}{3}{5}
	\end{tikzpicture}}
	=P(x).
	\]
	So the shape of the Young tableau $P(x)$ is $p(x)=(3,1,1,1)$.
\end{example}

For a Young diagram $P$, use $ (k,l) $ to denote the box in the $ k $-th row and the $ l $-th column.
We say the box $ (k,l) $ is \textit{even} (resp. \textit{odd}) if $ k+l $ is even (resp. odd). Let $ p_i ^{ev}$ (resp. $ p_i^{odd} $) be the number of even (resp. odd) boxes in the $ i $-th row of the Young diagram $ p $. 
One can easily check that \begin{equation}\label{eq:ev-od}
	p_i^{ev}=\begin{cases}
		\left\lceil \frac{p_i}{2} \right\rceil&\text{ if } i \text{ is odd},\\
		\left\lfloor \frac{p_i}{2} \right\rfloor&\text{ if } i \text{ is even},
	\end{cases}
	\quad p_i^{odd}=\begin{cases}
		\left\lfloor \frac{p_i}{2} \right\rfloor&\text{ if } i \text{ is odd},\\
		\left\lceil \frac{p_i}{2} \right\rceil&\text{ if } i \text{ is even}.
	\end{cases}
\end{equation}
Here for $ a\in \mathbb{R} $, $ \lfloor a \rfloor $ is the largest integer $ n $ such that $ n\leq a $, and $ \lceil a \rceil$ is the smallest integer $n$ such that $ n\geq a $. For convenience, we set
\begin{equation*}
	p^{ev}=(p_1^{ev},p_2^{ev},\cdots)\quad\mbox{and}\quad p^{odd}=(p_1^{odd},p_2^{odd},\cdots).
\end{equation*}

For $ x=(x_1,x_2,\cdots,x_n)\in \mathrm{Seq}_n (\Gamma) $, set
\begin{equation*}
	\begin{aligned}
		{x}^-=&(x_1,x_2,\cdots,x_{n-1}, x_n,-x_n,-x_{n-1},\cdots,-x_2,-x_1).
	\end{aligned}
\end{equation*}

\begin{Thm}[{\cite[Theorem 1.5]{BXX}}]\label{2}
	Let $\lambda+\rho=(\lambda_1, \lambda_2, \cdots, \lambda_n)\in \mathfrak{h}^*$ be  integral. Then
	
	\[	{\rm GKdim}\:L(\lambda)=\left\{
	\begin{array}{ll}
		\frac{n(n-1)}{2}-\sum\limits_{i\ge 1}(i-1)p(\lambda+\rho)_i &\text{~if~}\Delta =A_{n-1}\\	     	  	   
		n^2-\sum\limits_{i\ge 1}(i-1)p((\lambda+\rho)^-)_i^{odd} &\text{~if~}\Delta =B_{n}/C_{n}\\
		n^2-n-\sum\limits_{i\ge 1}(i-1)p((\lambda+\rho)^-)_i^{ev} &\text{~if~}\Delta =D_{n}
	\end{array}	
	\right.
	\]	
	
\end{Thm}

When $\lambda$ is non-integral,  we need some more notations before we give the algorithm of GK dimension.

We define three functions $ F_A $, $ F_B $, $ F_D $ as 
\begin{align*}
	F_A(x)&=\sum_{k\geq 1} (k-1) p_k,\\
	F_B(x)&=\sum_{k\geq 1} (k-1) p_k^{odd},\\
	F_D(x)&=\sum_{k\geq 1} (k-1) p_k^{ev},
\end{align*}
where $ p=p(x)=(p_1,p_2,\cdots) $.

\begin{definition}
	Fix $ \lambda+\rho=(\lambda_1,\cdots,\lambda_n) \in \mathfrak{h}^*$.
	
	For $ \mathfrak g= \gln$,  we define $[\lambda]$ to be the set of  maximal subsequences $ x $ of  $ \lambda+\rho $ such that any two entries of $ x $ has an integral difference. 
	
	For $\mathfrak{g}=\spn, \sob $ or $ \sod $, we define $[\lambda] $ to be the set of  maximal subsequences $ x $ of  $ \lambda+\rho $ such that any two entries of $ x $ have an integral  difference or sum. In this case, we set $ [\lambda]_1 $ (resp. $ [\lambda]_2 $) be to the subset of $ [\lambda] $ consisting of sequences with  all entries belonging to $ \mathbb{Z} $ (resp. $ \frac12+\mathbb{Z} $).
	Since there is at most one element in $[\lambda]_1 $ and $[\lambda]_2 $, we denote them by  $(\lambda+\rho)_{(0)}$ and $(\lambda+\rho)_{(\frac{1}{2})}$.
	We set $[\lambda]_{1,2}=[\lambda]_1\cup [\lambda]_2, \quad [\lambda]_3=[\lambda]\setminus[\lambda]_{1,2}$.
	
	\begin{example}
		Let $\lambda+\rho$=(7,5,3,3.5,2.5,1.5,2.3,1.3), then
		$$(\lambda+\rho)_{(0)}=(7,5,3),\:(\lambda+\rho)_{(\frac{1}{2})}=(3.5,2.5,1.5)\text{~and~}(\lambda+\rho)_{(0.3)}=(2.3,1.3)\in[\lambda]_3.$$
	\end{example}
	
\end{definition}

\begin{definition}
	Let  $ x=(\lambda_{i_1}, \lambda_{i_2},\cdots \lambda_{i_r})\in[\lambda]_3 $. Let $  y=(\lambda_{j_1}, \lambda_{j_2},\cdots, \lambda_{j_p}) $ be the maximal subsequence of $ x $ such that $ j_1=i_1 $ and the difference of any two entries of $ y$ is an integer. Let $ z= (\lambda_{k_1}, \lambda_{k_2},\cdots, \lambda_{k_q}) $ be the subsequence obtained by deleting $ y$ from $ x $, which is possible empty. 
	Define
	$$  \tilde{x}=(\lambda_{j_1}, \lambda_{j_2},\cdots, \lambda_{j_p}, -\lambda_{k_q}, -\lambda_{k_{q-1}},\cdots ,-\lambda_{k_1}).  $$
\end{definition}




\begin{Thm}[{{\cite[Theorem 4.6]{BX1}} and \cite[Theorem 5.7]{BXX}}]\label{GKdim}
	The GK dimension of  $ L(\lambda) $  can be computed as follows.
	\begin{enumerate}
		\item If $  \mathfrak{g}= \gln$, then
		\[
		{\rm GKdim}\:L(\lambda)=\frac{n(n-1)}{2}-\sum _{x\in [\lambda]} F_A(x).
		\]
		
		\item If $  \mathfrak{g}= \mathfrak{sp}(n,\mathbb{C})$, then
		\[
	{\rm GKdim}\:L(\lambda)=n^2- F_B((\lambda+\rho)_{(0)}^-)- F_D((\lambda+\rho)_{(\frac{1}{2})}^-)-\sum _{x\in [\lambda]_3} F_A(\tilde{x}).
		\]
		\item  If $  \mathfrak{g} = \mathfrak{so}(2n+1,\mathbb{C}) $, then
		\[
	{\rm GKdim}\:L(\lambda)=	n^2-F_B((\lambda+\rho)_{(0)}^-)-F_B((\lambda+\rho)_{(\frac{1}{2})}^-)-\sum _{x\in [\lambda]_3} F_A(\tilde{x}).
		\]
		\item  If $  \mathfrak{g} = \mathfrak{so}(2n,\mathbb{C}) $, then
		\[
		{\rm GKdim}\: L(\lambda)=n^2-n-F_D((\lambda+\rho)_{(0)}^-)-F_D((\lambda+\rho)_{(\frac{1}{2})}^-)-\sum _{x\in [\lambda]_3} F_A(\tilde{x}).
		\]
	\end{enumerate}
\end{Thm}

\section{Reducibility of scalar generalized Verma modules for classical Lie algebras}

Let $\mathfrak{g}$ be a finite-dimensional complex semisimple Lie
algebra and let $\mathfrak{b}=\mathfrak{h}\oplus\mathop\oplus\limits_{\alpha\in\Delta^+}\frg_\alpha$
be a fixed borel subalgebra of $\frg$. For a minimal parabolic subalgebra $\frq$, we appreciate that $\frq$ corresponds to the subsets  $\Pi\setminus\{\alpha_i\}_{i\ne p}$. It's easy to get that  $\lambda=z\eta$ for some $z\in\mathbb{C}$ and $\eta=\sum\limits_{i\neq p}k_i\omega_i$ by the Weyl dimension formula, where $\omega_i$ is the fundamental weight of simple root $\alpha_i$. In this paper, we suppose that $\eta=\sum\limits_{i\neq p}\omega_i=\widehat{\omega_p}$ and we call such a reducible point a {\it diagonal-reducible point}. 

	From Lemma \ref{GKdown}, the set of diagonal-reducible points of  a scalar generalized Verma module $M_I(z\widehat{\omega_p})$ is given in the following diagram:
	\begin{center}
		\setlength{\unitlength}{1mm}
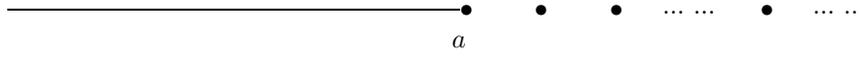
		\begin{picture}(180,4)
		
		\line(1,0){60}
		\put(0,-1){$\bullet $}
		\put(-1,-5){$a$}
		\put(10,-1){$\bullet $}
		\put(20,-1){$\bullet $}
		
		\put(27,-0.5){$...~~...$ }
		\put(40,-1){$\bullet $}
		\put(47,-0.5){$$...~~...$$}

		\end{picture}
		\hspace{2cm}
		
	\end{center}
	where the diagonal-reducible points starting from $z=a\in \mathbb{R}$ are equally spaced at an interval of length $1$ and are like the form $a+\mathbb{Z}_{\geq 0}$. The point $a$  will be called the {\it first diagonal-reducible point} of $M_I(z\widehat{\omega_p})$.
	
	From  Lemma \ref{GKdown}, we only need to find the first diagonal-reducible point of the scalar generalized Verma module $M_{I}(z\widehat{\omega_p})$.

\begin{lem}\label{5}
	Let $w_i$ be the fundamental weight of $\alpha_i$ and $\rho$ be half the sum of the members of $\Delta^+$, then $\rho=\sum\limits_{i\ge 1}w_i$.
\end{lem}
\begin{proof}
	In \cite{KN}, we can get this conclusion case by case.
\end{proof}

\begin{lem}
	If $\frq$ is the minimal parabolic subalgebra, then $\dim(\fru)=\arrowvert\Delta^+\arrowvert-1$.
\end{lem}
\begin{proof}
	We know $\frq=\frl\oplus\fru$ and $\dim(\Delta^+(\frl))=1$, then the value of $\dim(\fru)$ is obvious.
\end{proof}

\begin{Rem}From Theorem \ref{GKdim}, we know that the GK dimension of a highest weight module $L(\lambda)$ only depends on the shape of some Young tableaux associated with $\lambda$. 	Sometimes the entries $\lambda_i$ in $ \lambda+\rho=(\lambda_1,\cdots,\lambda_n) \in \mathfrak{h}^*$ are complicated and if there is no ambiguity we usually use the positions of their entries to represent them in our Young tableau.
	
	For example, the Young tableau for
	$\lambda+\rho=(\lambda_1,\lambda_2,\lambda_3,\lambda_4)=(3,1,4,2)$ is 	\[P(\lambda+\rho)=
	{\begin{tikzpicture}[scale=\domscale+0.1,baseline=-19pt]
			\hobox{0}{0}{1}
			\hobox{1}{0}{2}
			\hobox{0}{1}{3}
			\hobox{1}{1}{4}
	\end{tikzpicture}}.
	\]
	We will use \[{\begin{tikzpicture}[scale=\domscale+0.1,baseline=-19pt]
			\hobox{0}{0}{2}
			\hobox{1}{0}{4}
			\hobox{0}{1}{1}
			\hobox{1}{1}{3}
	\end{tikzpicture}}\]
	to represent our Young tableau $P(\lambda+\rho)$.
\end{Rem}

\begin{prop}\label{4}
	Let $ \mathfrak{g}=\mathfrak{sl}(n,\mathbb{C}) $. $M_{I}(\lambda)$ is reducible if and only if 
\begin{enumerate}
	\item If $p=1$ or $p=n-1$, then $z\in \mathbb{Z}_{\ge 0}$.
	
	\item If $2\leq p\leq n-2$, then $z\in -1+\mathbb{Z}_{\ge 0}$.	
\end{enumerate}		
\end{prop}
\begin{proof}
	Take $\mathfrak{g}=\mathfrak{sl}(\frn,\mathbb{C})$, and $\Delta^+(\frl)=\{\alpha_{p}\}$,$\alpha_p=e_p-e_{p+1}$.
	
	When $M_I(\lambda)$ is of scalar type, we know that $\lambda=z\eta $ for some $z\in\mathbb{R}$,
	and $\eta=(\underbrace{\frac{n-1}{2}-\frac{n-p}{n},\dots,\frac{n-(2p-1)}{2}-\frac{n-p}{n}}_{p},\underbrace{\frac{n-(2p+1)}{2}+\frac{p}{n},\dots,\frac{1-n}{2}+\frac{p}{n}}_{n-p})$.
	
	In \cite{EHW}, we know that $2\rho=(n-1,n-3,\dots,-n+3,-n+1)$, thus
	\begin{align*}
		\rho=(\frac{n-1}{2},\frac{n-3}{2},\dots,\frac{-n+3}{2},\frac{-n+1}{2}). 
	\end{align*}
\begin{enumerate}
	\item If $z\in\mathbb{Z}$, we will have the follows.
	    \begin{enumerate}
	    	\item When $z\ge 0$, then
	    	
	    	$\lambda+\rho$ is decreasing, so Young tableau has one column, by Lemma \ref{2}
	    	\begin{align}\label{3}
	    	{\rm GKdim}\: L(\lambda)\notag&=\frac{n(n-1)}{2}-(1+2+\dots+n-1)\notag\\&=\frac{n(n-1)}{2}-\frac{n(n-1)}{2}=0<\dim(\mathfrak{u}).
	    	\end{align}
	    By Lemma \ref{reducible} we obtain that $M_{I}(\lambda)$ is reducible.\\	
	    	
	    	\item When $z=-1$, then
	    	$$
	    		\lambda+\rho=(\underbrace{\frac{n-p}{n},\dots,\frac{n-p}{n}}_{p},\underbrace{-\frac{p}{n},\dots,\frac{p}{n}}_{n-p}).
	    	$$
	    	Whether $n-p\ge p$ or $n-p<p$
	    	$$
	    		p(\lambda+\rho)=(n-p,p).
	    	$$
	    	 By Lemma \ref{2}
	    $$
	    	{\rm GKdim}\:L(\lambda)=\frac{n(n-1)}{2}-(0\cdot (n-p)+p)=\frac{n(n-1)}{2}-p.
	    $$
If $p=1$ or $p=n-1$, ${\rm GKdim}\:L(\lambda)=\dim(\mathfrak{u})$, by Lemma \ref{reducible} $M_{I}(\lambda)$ is irreducible.

If $2\leq p \leq n-2$, ${\rm GKdim}\:L(\lambda)<\dim(\mathfrak{u})$, by Lemma \ref{reducible} $M_{I}(\lambda)$ is reducible.
	    	
	    	\item When $z=-2$, then
$$
	\lambda+\rho=(\underbrace{-\frac{n}{2}-\frac{2p}{n}+\frac{5}{2},\dots,-\frac{n}{2}-\frac{2p}{n}+\frac{2p+3}{2},}_{p}-\frac{n}{2}-\frac{2p}{n}+\frac{2p+1}{2},\dots,\frac{n}{2}-\frac{2p}{n}-\frac{1}{2}).
$$	    	
\[
\tiny{\begin{tikzpicture}[scale=\domscale+0.25,baseline=-15pt]
	\hobox{0}{0}{1}
	\hobox{1}{0}{2}
	\hobox{2}{0}{\cdots}
	\hobox{3}{0}{p} 
\end{tikzpicture}\to 
\begin{tikzpicture}[scale=\domscale+0.25,baseline=-20pt]
	\hobox{0}{0}{1}
    \hobox{1}{0}{2}
    \hobox{2}{0}{\cdots}
    \hobox{3}{0}{n}
    \hobox{0}{1}{p}
\end{tikzpicture}}=P(\lambda)
\]	 
$$
	\lambda+\rho=(n-1,1),
$$
and
\begin{align*}
	{\rm GKdim}\:L(\lambda)&=\frac{n(n-1)}{2}-1=\dim(\mathfrak{u}).
\end{align*}	 	
	    \end{enumerate}

 By Lemma \ref{reducible} $M_{I}(\lambda)$ is irreducible. \\
 
    \item If $z\notin\mathbb{Z}$, we have the follows.
    
The difference between the first $p$ components of $\lambda+\rho$ and the last $n-p$ components of $\lambda+\rho$ is $z+1$. Then
$$
{\rm GKdim}\: L(\lambda)=\frac{n(n-1)}{2}-\frac{2(2-1)}{2}=\frac{n(n-1)}{2}-1=\dim(\mathfrak{u}).
$$  	
\end{enumerate}

By Lemma \ref{reducible} $M_{I}(\lambda)$ is irreducible. And we have completed the proof of the proposition \ref{4}.
\end{proof}

\begin{prop}\label{7}
	Let $ \mathfrak{g}=\mathfrak{so}(2n+1,\mathbb{C})$  $(n>2)$. $M_{I}(\lambda)$ is reducible if and only if
\begin{enumerate}
	\item When $z\in\mathbb{Z}$, then
	\begin{enumerate}
	\item If $p=1$, then $z\in\mathbb{Z}_{\ge 0}$.
	
	\item If $1<p\le n$, then $z\in -1+\mathbb{Z}_{\ge 0}$.	
\end{enumerate}
\item When $z\in\dfrac{1}{2}+\mathbb{Z}$, then
\begin{enumerate}	
	\item If $n$ is even, then $z\in-\dfrac{1}{2}+\mathbb{Z}_{\ge 0}$.
	
	\item If $n$ is odd, then $z\in\dfrac{1}{2}+\mathbb{Z}_{\ge 0}$ for $n=3$ and $z\in-\dfrac{1}{2}+\mathbb{Z}_{\ge 0}$ for $n>3$.
\end{enumerate}

\end{enumerate}
\end{prop}
\begin{proof}
	Take $\mathfrak{g}=\mathfrak{so}(2n+1,\mathbb{C})$, and $\Delta^+(\frl)=\{\alpha_{p}\}$.
\begin{enumerate}
	\item If $p<n$, then by Lemma \ref{5} we can get that
\begin{align*}
	\eta&=\rho-w_p\\&=(n-\frac{3}{2},\dots,n-p-\frac{1}{2},n-p-\frac{1}{2},\dots,\frac{1}{2}),
\end{align*}	
\begin{align*}
	\lambda+\rho=((n-\frac{3}{2})z+n-\frac{1}{2},\dots,(n-p-\frac{1}{2})z+n-p+\frac{1}{2},(n-p-\frac{1}{2})z+n-p-\frac{1}{2},\dots,\frac{1}{2}z+\frac{1}{2}).
\end{align*}
\begin{enumerate}
	\item If $z\in\mathbb{Z}$, we will have the follows.
	    \begin{enumerate}
	    	\item When $z\ge 0$, $M_{I}(\lambda)$ is reducible by (\ref{3}).
	    
	     	\item When $z=-1$, then
	    \begin{align*}
	    	(\lambda+\rho)^-=(\underbrace{1,\dots,1}_{p},\underbrace{0,0,\dots,0}_{2n-2p},\underbrace{-1,\dots,-1}_{p}).
	    \end{align*}
    
    \begin{enumerate}
    	\item $n-p\ge p$ and $p$ is even, we have
    $$
    p((\lambda+\rho)^-)^{odd}=(n-p,\frac{p}{2},\frac{p}{2}),
    $$
$$
{\rm GKdim}\:L(\lambda)=n^2-(0\cdot (n-p)+\frac{p}{2}+2\cdot \frac{p}{2})=n^2-\frac{3p}{2}<n^2-1=\dim(\mathfrak{u}).
$$ 
   	
    \item $n-p\ge p$ and $p$ is odd, we have
$$
    	p((\lambda+\rho)^-)^{odd}=(n-p,\frac{p+1}{2},\frac{p-1}{2}),
  $$
$$
    {\rm GKdim}\:L(\lambda)=n^2-(0\cdot (n-p)+\frac{p+1}{2}+2\cdot \frac{p-1}{2})=n^2-\frac{3p}{2}+\frac{1}{2}.
 $$ 
 	
    If $p=1$, ${\rm GKdim}\:L(\lambda)=\dim(\mathfrak{u})$. 
    
    If $p>1$, ${\rm GKdim}\:L(\lambda)<\dim(\mathfrak{u})$. By Lemma \ref{reducible} we can get the conclusion.
    	
	\item $n-p<p$ and $p$ is even, we have    
	 \begin{align*}
		p((\lambda+\rho)^-)^{odd}=(\frac{p}{2},\frac{p}{2},n-p),
	\end{align*}
	\begin{align*}
		{\rm GKdim}\:L(\lambda)&=n^2-(0\cdot \frac{p}{2}+1\cdot \frac{p}{2}+2(n-p))\\&=n^2-\frac{p}{2}-2(n-p)<n^2-1=\dim(\mathfrak{u}).
	\end{align*}      
	    
	    \item $n-p<p$ and $p$ is odd, we have	
	     \begin{align*}
	    	p((\lambda+\rho)^-)^{odd}=(\frac{p-1}{2},\frac{p+1}{2},n-p),
	    \end{align*}
	    \begin{align*}
	    	{\rm GKdim}\:L(\lambda)&=n^2-(0\cdot \frac{p-1}{2}+1\cdot \frac{p+1}{2}+2(n-p))\\&=n^2-\frac{p+1}{2}-2(n-p)<n^2-1=\dim(\mathfrak{u}).
	    \end{align*} 
    \end{enumerate}
    \item When $z=-2$, then
    \begin{align*}
    	\lambda+\rho=(-n+\frac{5}{2},\dots,-n+p+\frac{1}{2},\dots,-\frac{1}{2}).
    \end{align*}
    \[
    \tiny{\begin{tikzpicture}[scale=\domscale+0.25,baseline=-15pt]
    	\hobox{0}{0}{1}
    	\hobox{1}{0}{2}
    	\hobox{2}{0}{\cdots}
    	\hobox{3}{0}{p} 
    \end{tikzpicture}\to 
    \begin{tikzpicture}[scale=\domscale+0.25,baseline=-20pt]
    	\hobox{0}{0}{1}
    	\hobox{1}{0}{2}
    	\hobox{2}{0}{\cdots}
    	\hobox{3}{0}{n}
    	\hobox{0}{1}{p}
    \end{tikzpicture}\to 
    \begin{tikzpicture}[scale=\domscale+0.25,baseline=-20pt]
        \hobox{0}{0}{1}
        \hobox{1}{0}{2}
        \hobox{2}{0}{\cdots}
        \hobox{3}{0}{2n}
        \hobox{0}{1}{p}
        \hobox{1}{1}{2n-p}
    \end{tikzpicture}}=P(\lambda)
    \]	 
 $$
    	p((\lambda+\rho)^-)^{odd}=(n-1,1),
$$

\end{enumerate}

and ${\rm GKdim} L(\lambda)=n^2-(0\cdot (n-1)+1\cdot 1=n^2-1=\dim(\mathfrak{u}).$ Hence $z=-2$ is an irreducible point.\\  

\item $z\in\dfrac{1}{2}+\mathbb{Z}$, we will have the follows.
\begin{enumerate}
	\item When $z> -\dfrac{1}{2}$, we can get that $M_{I}(\lambda)$ is reducible.
	
	\item When $z=-\dfrac{1}{2}$, we will have 
	$$(\lambda+\rho)_{(z)}=\lambda+\rho=(\dfrac{1}{2}n+\dfrac{1}{4},\dfrac{1}{2}n-\dfrac{1}{4},\cdots,\dfrac{1}{2}n-\dfrac{1}{2}p+\dfrac{3}{4},\dfrac{1}{2}n-\dfrac{1}{2}p-\dfrac{1}{4},\cdots,\dfrac{1}{4}).$$
		We divide the discussion into four cases:
\begin{enumerate}	
	\item $n$ is even and $p$ is odd, then
	$$(\lambda+\rho)_{(z_1)}=(\dfrac{1}{2}n+\dfrac{1}{4},\dfrac{1}{2}n-\dfrac{3}{4},\cdots,\dfrac{1}{2}n-\dfrac{1}{2}p+\dfrac{3}{4},\dfrac{1}{2}n-\dfrac{1}{2}p-\dfrac{1}{4},\cdots,\dfrac{1}{4}).$$		
		$$(\lambda+\rho)_{(z_2)}=(\dfrac{1}{2}n-\dfrac{1}{4},\dfrac{1}{2}n-\dfrac{5}{4},\cdots,\dfrac{1}{2}n-\dfrac{1}{2}p+\dfrac{5}{4},\dfrac{1}{2}n-\dfrac{1}{2}p-\dfrac{3}{4},\cdots,\dfrac{3}{4}).$$
		\begin{align}\label{8}
			{\rm GKdim}\:L(\lambda)\notag&=n^2-(1+2+\cdots+\dfrac{1}{2}n)-(1+2+\cdots+\dfrac{1}{2}n-2)\\&=n^2-(\dfrac{1}{4}n^2-\dfrac{1}{2}n)-1.
		\end{align}	
	
\item $n$ is even and $p$ is even, then	
	$$(\lambda+\rho)_{(z_1)}=(\dfrac{1}{2}n+\dfrac{1}{4},\dfrac{1}{2}n-\dfrac{3}{4},\cdots,\dfrac{1}{2}n-\dfrac{1}{2}p+\dfrac{5}{4},\dfrac{1}{2}n-\dfrac{1}{2}p-\dfrac{3}{4},\cdots,\dfrac{1}{4}).$$		
$$(\lambda+\rho)_{(z_2)}=(\dfrac{1}{2}n-\dfrac{1}{4},\dfrac{1}{2}n-\dfrac{5}{4},\cdots,\dfrac{1}{2}n-\dfrac{1}{2}p+\dfrac{3}{4},\dfrac{1}{2}n-\dfrac{1}{2}p-\dfrac{1}{4},\cdots,\dfrac{3}{4}).$$	
\begin{align*}
	{\rm GKdim}\:L(\lambda)&=n^2-2(1+2+\cdots+\dfrac{1}{2}n-1)\\&=n^2-(\dfrac{1}{4}n^2-\dfrac{1}{2}n).
\end{align*}			
	
\item $n$ is odd and $p$ is odd, then	
$$(\lambda+\rho)_{(z_1)}=(\dfrac{1}{2}n+\dfrac{1}{4},\dfrac{1}{2}n-\dfrac{3}{4},\cdots,\dfrac{1}{2}n-\dfrac{1}{2}p+\dfrac{3}{4},\dfrac{1}{2}n-\dfrac{1}{2}p-\dfrac{1}{4},\cdots,\dfrac{1}{4}).$$		
$$(\lambda+\rho)_{(z_2)}=(\dfrac{1}{2}n-\dfrac{1}{4},\dfrac{1}{2}n-\dfrac{5}{4},\cdots,\dfrac{1}{2}n-\dfrac{1}{2}p+\dfrac{5}{4},\dfrac{1}{2}n-\dfrac{1}{2}p-\dfrac{3}{4},\cdots,\dfrac{1}{4}).$$	
\begin{align*}
	{\rm GKdim}\:L(\lambda)&=n^2-(1+2+\cdots+\dfrac{1}{2}n-\dfrac{1}{2})-(1+2+\cdots+\dfrac{1}{2}n-\dfrac{3}{2})\\&=n^2-\dfrac{1}{4}(n-1)^2.
\end{align*}

If $n=3$, ${\rm GKdim}\:L(\lambda)=\dim(\mathfrak{u})$. 

If $n>3$, ${\rm GKdim}\:L(\lambda)<\dim(\mathfrak{u})$. 

\item $n$ is odd and $p$ is even, then
$$(\lambda+\rho)_{(z_1)}=(\dfrac{1}{2}n+\dfrac{1}{4},\dfrac{1}{2}n-\dfrac{3}{4},\cdots,\dfrac{1}{2}n-\dfrac{1}{2}p+\dfrac{5}{4},\dfrac{1}{2}n-\dfrac{1}{2}p-\dfrac{3}{4},\cdots,\dfrac{3}{4}).$$		
$$(\lambda+\rho)_{(z_2)}=(\dfrac{1}{2}n-\dfrac{1}{4},\dfrac{1}{2}n-\dfrac{5}{4},\cdots,\dfrac{1}{2}n-\dfrac{1}{2}p+\dfrac{3}{4},\dfrac{1}{2}n-\dfrac{1}{2}p-\dfrac{1}{4},\cdots,\dfrac{1}{4}).$$	
\begin{align*}
	{\rm GKdim}\:L(\lambda)&=n^2-(1+2+\cdots+\dfrac{1}{2}n-\dfrac{1}{2})-(1+2+\cdots+\dfrac{1}{2}n-\dfrac{3}{2})\\&=n^2-\dfrac{1}{4}(n-1)^2.
\end{align*}
	
\end{enumerate}	
\item When $z=-\dfrac{3}{2}$, we will have
$$\lambda+\rho=(-\dfrac{1}{2}n+\dfrac{7}{4},-\dfrac{1}{2}n+\dfrac{9}{4},\cdots,-\dfrac{1}{2}n+\dfrac{1}{2}p+\dfrac{5}{4},-\dfrac{1}{2}n+\dfrac{1}{2}p+\dfrac{1}{4},\cdots,\dfrac{1}{4}).$$	
	Also we will divide the discussion into four cases:
\begin{enumerate}	
\item If $n$ is even and $p$ is odd, then
$$ (\lambda+\rho)_{(z_1)}=(-\dfrac{1}{2}n+\dfrac{7}{4},-\dfrac{1}{2}n+\dfrac{11}{4},\cdots,-\dfrac{1}{2}n+\dfrac{1}{2}p+\dfrac{5}{4},-\dfrac{1}{2}n+\dfrac{1}{2}p+\dfrac{1}{4},\cdots,\dfrac{1}{4}).$$
$$ (\lambda+\rho)_{(z_2)}=(-\dfrac{1}{2}n+\dfrac{9}{4},-\dfrac{1}{2}n+\dfrac{13}{4},\cdots,-\dfrac{1}{2}n+\dfrac{1}{2}p+\dfrac{3}{4},-\dfrac{1}{2}n+\dfrac{1}{2}p+\dfrac{3}{4},\cdots,-\dfrac{3}{4}).$$
$${\rm GKdim}\:L(\lambda)=n^2-1=\dim(\fru).$$

\item If $n$ is even and $p$ is even, then	
$$p(\lambda+\rho)_{(z_1)}=(\dfrac{1}{2}n)\text{~and~}p(\lambda+\rho)_{(z_2)}=(\dfrac{1}{2}n-1,1).$$
$${\rm GKdim}\:L(\lambda)=n^2-1=\dim(\fru).$$

\item  If $n$ isodd and $p$ is even, then
$$p(\lambda+\rho)_{(z_1)}=(\dfrac{1}{2}n-\dfrac{1}{2})\text{~and~}p(\lambda+\rho)_{(z_2)}=(\dfrac{1}{2}n-\dfrac{1}{2},1).$$
$${\rm GKdim}\:L(\lambda)=n^2-1=\dim(\fru).$$		
\end{enumerate}

All in all, $z=-\dfrac{3}{2}$ is an irreducible point.
\end{enumerate}

	\item If $z\notin\mathbb{Z}$ and $z\notin\dfrac{1}{2}+\mathbb{Z}$, then	
	\begin{align}\label{6}
		{\rm GKdim}\:L(\lambda)=n^2-\frac{2\cdot 1}{2}=n^2-1=\dim(\mathfrak{u}).
	\end{align}   	
\end{enumerate}

In this case, $M_{I}(\lambda)$ is irreducible.\\

	\item If $p=n$, then
\begin{align*}
	\lambda+\rho=((n-1)z+n-\frac{1}{2},(n-2)z+n-\frac{3}{2},\dots,\frac{1}{2}).
\end{align*}	
\begin{enumerate}
	\item If $z\in\mathbb{Z}$, we will have the follows.
\begin{enumerate}
	\item When $z\ge 0$, then $M_{I}(\lambda)$ is reducible by (\ref{3}).

    \item When $z=-1$, then
$$
	(\lambda+\rho)^-=(\underbrace{\frac{1}{2},\dots,\frac{1}{2}}_{p},\underbrace{-\frac{1}{2},\dots,-\frac{1}{2}}_{p}),
$$	
and
$$
	p(\lambda+\rho)^-=(n,n).
$$	
\begin{enumerate}
		\item If $n$ is even, then
$$
	p((\lambda+\rho)^-)^{odd}=(\frac{n}{2},\frac{n}{2}),
$$	
$$
	{\rm GKdim}\: L(\lambda)=n^2-(0\cdot\frac{n}{2}+\frac{n}{2})=n^2-\frac{n}{2}<\dim(\mathfrak{u}).
$$

    \item If $n$ is odd, then
$$
	p((\lambda+\rho)^-)^{odd}=(\frac{n-1}{2},\frac{n+1}{2}),
$$	
$$
	{\rm GKdim}\: L(\lambda)=n^2-(0\cdot\frac{n-1}{2}+\frac{n+1}{2})=n^2-\frac{n+1}{2}<\dim(\mathfrak{u}).
$$
\end{enumerate}
All in all, $z=-1$ is the diagonal-reducible point.\\

    \item When $z=-2$, then
$$
	(\lambda+\rho)^-=(-n+\frac{3}{2},-n+\frac{5}{2},\dots,\frac{1}{2},-\frac{1}{2},\dots,n-\frac{5}{2},n-\frac{3}{2}),
$$	
$$
	p((\lambda+\rho)^-)^{odd}=(n-1,1),
$$	
$$
	{\rm GKdim}\:L(\lambda)=n^2-(0\cdot (n-1)+1)=n^2-1=\dim(\mathfrak{u}).
$$

All in all, $z=-1$ is the irreducible point.				
\end{enumerate}	
\item $z\in\dfrac{1}{2}+\mathbb{Z}$, the process is similar to the case of $p = n$, and we omit this part of the process.

\item If $z\notin\mathbb{Z}$ and $z\notin\dfrac{1}{2}+\mathbb{Z}$, $M_{I}(\lambda)$ is irreducible by (\ref{6}).	
\end{enumerate}	
\end{enumerate}	
	
This concludes the proof of Proposition \ref{7}.	
\end{proof}

\begin{prop}
		Let $ \mathfrak{g}=\mathfrak{sp}(2n,\mathbb{C})$ $(n>2)$. $M_{I}(\lambda)$ is reducible if and only if
\begin{enumerate}
	\item When $z\in\mathbb{Z}$, then
	\begin{enumerate}	
	\item If $p=1$, then $z\in\mathbb{Z}_{\ge 0}$.
	
	\item If $1<p\le n$, then $z\in -1+\mathbb{Z}_{\ge 0}$.		
    \end{enumerate}
    \item When $z\in\dfrac{1}{2}+\mathbb{Z}$, then	$z\in-\dfrac{1}{2}+\mathbb{Z}_{\ge 0}.$
    
\end{enumerate}		
\end{prop}
\begin{proof}
	When $z\in\mathbb{Z}$, the process of type $B_n$ and type $C_n$ is the same, and we omit this part of the process and see the details in Proposition \ref{7}.
	
	For the case when $z\in\dfrac{1}{2}+\mathbb{Z}$, we can get that 
	\[
	{\rm GKdim}\:L(\lambda)=\left\{
	\begin{array}{ll}
		n^2-(\frac{1}{2}n^2-\frac{1}{2}n)-1 &\text{if $p$ is  odd}\\	    	  	
		n^2-(\frac{1}{2}n^2-\frac{1}{2}n) &\text{if $p$ is even}	
	\end{array}	
	\right.
	\]
for $z=-\dfrac{1}{2}$ and ${\rm GKdim}\:L(\lambda)=n^2-1$ for $z=-\dfrac{3}{2}$. It is not difficult to draw the conclusion.

When $z\notin\mathbb{Z}$ and $z\notin\dfrac{1}{2}+\mathbb{Z}$, $M_{I}(\lambda)$ is irreducible by (\ref{6}).

\end{proof}

\begin{prop}
		Let $ \mathfrak{g}=\mathfrak{so}(2n,\mathbb{C})$ $(n>2)$. $M_{I}(\lambda)$ is reducible if and only if
		\begin{enumerate}
			\item If $1\le p\le n-2$, then $z\in-1+\dfrac{1}{2}\mathbb{Z}_{\ge 0}$.
			
			\item If $p=n-1{\rm~or~}n$,
			    \begin{enumerate}
			        \item $n=3$, then $z\in\dfrac{1}{2}\mathbb{Z}_{\ge 0}$.	
			    	
			    	\item $n>3$, then $z\in-1+\dfrac{1}{2}\mathbb{Z}_{\ge 0}$.			    	
			    \end{enumerate}			
		\end{enumerate}
\end{prop}  
 \begin{proof}
 		Take $\mathfrak{g}=\mathfrak{so}(2n,\mathbb{C})$, and $\Delta^+(\frl)=\{\alpha_{p}\}$.
 \begin{enumerate}	
 	\item If $1\le p\le n-2$, we have
 	\begin{align*}
 		\eta&=\rho-w_p\\&=(n-2,\dots,n-p-1,n-p-1,\dots,0).
 	\end{align*}
 $$
 		\lambda+\rho=((n-2)z+n-1,\dots,(n-p-1)z+n-p,(n-p-1)z+n-p-1,\dots,0).
$$
 \begin{enumerate}
 	\item If $z\in\mathbb{Z}$, we will have the follows.
       \begin{enumerate}
       	    \item When $z\ge 0$, then $M_{I}(\lambda)$ is reducible by (\ref{3}).
       	    
       	    \item When $z=-1$, then
       	    $$
       	    	(\lambda+\rho)^-=(\underbrace{1,\dots,1}_{p},\underbrace{0,0,\dots,0}_{2n-2p},\underbrace{-1,\dots,-1}_{p}).
       	    $$
       	    \begin{enumerate}
       	\item $n-p\ge p$ and $p$ is even, then
       	$$
       		p((\lambda+\rho)^-)^{ev}=(n-p,\frac{p}{2},\frac{p}{2}),
       	$$
       	\begin{align*}
       		{\rm GKdim}\:L(\lambda)&=n^2-n-(0\cdot (n-p)+\frac{p}{2}+2\cdot \frac{p}{2})\\&=n^2-n-\frac{3p}{2}<n^2-n-1=\dim(\mathfrak{u}).
       	\end{align*} 
       	
       	\item $n-p\ge p$ and $p$ is odd, then
       		$$
       		p((\lambda+\rho)^-)^{ev}=(n-p,\frac{p-1}{2},\frac{p+1}{2}),
       $$
       	\begin{align*}
       		{\rm GKdim}\:L(\lambda)&=n^2-n-(0\cdot (n-p)+\frac{p-1}{2}+2\cdot \frac{p+1}{2})\\&=n^2-n-\frac{3p+1}{2}<n^2-n-1=\dim(\mathfrak{u}).
       	\end{align*} 
       	
       	\item $n-p< p$ and $p$ is even, then
       	$$
       		p((\lambda+\rho)^-)^{ev}=(\frac{p}{2},\frac{p}{2},n-p),
       $$
       	\begin{align*}
       		{\rm GKdim}\:L(\lambda)&=n^2-n-(0\cdot \frac{p}{2}+ \frac{p}{2}+2(n-p))\\&=n^2-n-\frac{p}{2}-2(n-p)<n^2-n-1=\dim(\mathfrak{u}).
       	\end{align*} 
       
       	\item $n-p< p$ and $p$ is odd, then
       $$
       	p((\lambda+\rho)^-)^{ev}=(\frac{p+1}{2},\frac{p-1}{2},n-p),
       $$
       \begin{align*}
       	{\rm GKdim}\:L(\lambda)&=n^2-n-(0\cdot \frac{p+1}{2}+ \frac{p-1}{2}+2(n-p))\\&=n^2-n-\frac{p-1}{2}-2(n-p)<n^2-n-1=\dim(\mathfrak{u}).
       \end{align*} 
   \end{enumerate}
   Hence $z=-1$ is a diagonal-reducible point.\\
   
   \item When $z=-2$, then
$$
   	(\lambda+\rho)^-=(-n+3,\dots,-n+p+2,-n+p+1,\dots,0,0,\dots,n-p-1,n-p-2,\dots,n-3).
$$
 $$p(\lambda+\rho)^-=(2n-2,2),$$  
 and ${\rm GKdim}\:L(\lambda)=n^2-n-(0\cdot(n-1)+1)=n^2-n-1=\dim(\mathfrak{u}).$

         \end{enumerate} 
     
     \item If $z\in\dfrac{1}{2}+\mathbb{Z}$, we will have the follows.     
   \begin{enumerate}  
     \item When $z>-\dfrac{1}{2}$, $M_{I}(\lambda)$ is reducible.
     
     \item When $z=-\dfrac{1}{2}$, then
 $$\lambda+\rho-(\dfrac{1}{2}n,\cdots,\dfrac{1}{2}n-\dfrac{1}{2}p+\dfrac{1}{2},\dfrac{1}{2}n-\dfrac{1}{2}p-\dfrac{1}{2},\cdots,0).$$    
   	We divide the discussion into four cases:
   	  \begin{enumerate}  
     \item If $n$ is even and $p$ is odd, then
     $$(\lambda+\rho)_{(0)}=(\dfrac{1}{2}n,\dfrac{1}{2}-1,\cdots,\dfrac{1}{2}n-\dfrac{1}{2}p+\dfrac{1}{2},\dfrac{1}{2}n-\dfrac{1}{2}p-\dfrac{1}{2},\cdots,0).$$
     $$(\lambda+\rho)_{(\frac{1}{2})}=(\dfrac{1}{2}n-\dfrac{1}{2},\cdots,\dfrac{1}{2}-\dfrac{1}{2}p,\dfrac{1}{2}n-\dfrac{1}{2}p+-1,\cdots,\dfrac{1}{2}).$$
     $$p(\lambda+\rho)^-_{(0)}=(2,\underbrace{1,\cdots,1}_n)\text{~and~}p(\lambda+\rho)^-_{(\frac{1}{2})}=(\underbrace{1,\cdots,1}_{n-2}).$$
  \begin{align*}
  	{\rm GKdim}\:L(\lambda)&=n^2-n-(2+4+\cdots+n)-(2+4+\cdots+n-4)\\&=n^2-n-(\dfrac{1}{2}n^2-n+2).
  \end{align*}   
     
    \item If $n$ is even and $p$ is even, then 
      $$p(\lambda+\rho)^-_{(0)}=(2,\underbrace{1,\cdots,1}_{n-2})\text{~and~}p(\lambda+\rho)^-_{(\frac{1}{2})}=(\underbrace{1,\cdots,1}_{n}).$$
      \begin{align*}
     	{\rm GKdim}\:L(\lambda)&=n^2-n-(2+4+\cdots+n-2)-(2+4+\cdots+n-2)\\&=n^2-n-(\dfrac{1}{2}n^2-n).
     \end{align*}  
 
  \item If $n$ is odd and $p$ is even, then 
 $$p(\lambda+\rho)^-_{(0)}=(2,\underbrace{1,\cdots,1}_{n-1})\text{~and~}p(\lambda+\rho)^-_{(\frac{1}{2})}=(\underbrace{1,\cdots,1}_{n-1}).$$
 \begin{align*}
 	{\rm GKdim}\:L(\lambda)&=n^2-n-(2+4+\cdots+n-1)-(2+4+\cdots+n-2)\\&=n^2-n-\dfrac{1}{2}(n-1)^2.
 \end{align*}  
 
 \item If $n$ is odd and $p$ is odd, then 
 $$p(\lambda+\rho)^-_{(0)}=(2,\underbrace{1,\cdots,1}_{n-3})\text{~and~}p(\lambda+\rho)^-_{(\frac{1}{2})}=(\underbrace{1,\cdots,1}_{n+1}).$$
 \begin{align*}
 	{\rm GKdim}\:L(\lambda)&=n^2-n-(2+4+\cdots+n-3)-(2+4+\cdots+n-1)\\&=n^2-n-\dfrac{1}{2}(n-1)^2.
 \end{align*}    
      \end{enumerate} 
 \item When $z=-\dfrac{3}{2}$, we can easily check that  $	{\rm GKdim}\:L(\lambda)=n^2-n-1=\dim\fru.$ Then $z=-\dfrac{3}{2}$ is an irreducible point.

   \end{enumerate}
  
   \item If $z\notin\mathbb{Z}$ and $z\notin\dfrac{1}{2}+\mathbb{Z}$, we have $M_{I}(\lambda)$ is irreducible  by (\ref{6}). 
   
   All in all, $M_{I}(\lambda)$ is reducible if and only if $z\in -1+\mathbb{Z}_{\ge 0}$.\\
 \end{enumerate}
	
 	\item If $p=n-1$, then
 	\begin{align*}
 		\eta=\rho-w_{n-1}=(n-\frac{3}{2},n-\frac{5}{2},\dots,\frac{1}{2},\frac{1}{2}),
 	\end{align*}
$$
 	\lambda+\rho=((n-\frac{3}{2})z+n-1,(n-\frac{5}{2})z+n-2,\dots,\frac{1}{2}z+1,\frac{1}{2}z).
 $$
\begin{enumerate}
	\item If $z\in\mathbb{Z}$, we will have the follows.
	\begin{enumerate}
		\item When $z\ge 0$, then $M_{I}(\lambda)$ is reducible by (\ref{3}).
		
		\item When $z=-1$, then
	$$
			(\lambda+\rho)^-=(\underbrace{\frac{1}{2},\dots,\frac{1}{2}}_{n-1}-\frac{1}{2},\frac{1}{2},\underbrace{-\frac{1}{2},\dots,-\frac{1}{2}}_{n-1}).
	$$
	\begin{enumerate}
			\item When $n$ is even, then
			$$
				p((\lambda+\rho)^-)^{ev}=(\frac{n}{2},\frac{n}{2}),
		$$
			$$
				{\rm GKdim}\:L(\lambda)=n^2-n-(0\cdot \frac{n}{2}+ \frac{n}{2})=n^2-n-\frac{n}{2}<n^2-n-1=\dim(\mathfrak{u}).
		$$ 
		
		    \item When $n$ is odd, then
		   $$
		    	p((\lambda+\rho)^-)^{ev}=(\frac{n+1}{2},\frac{n-1}{2}),
		   $$
		 $$
		    	{\rm GKdim}\:L(\lambda)=n^2-n-(0\cdot \frac{n+1}{2}+ \frac{n-1}{2})=n^2-n-\frac{n-1}{2}.
		   $$
	    \end{enumerate}
    
			If $n=3$, ${\rm GKdim}\:L(\lambda)=\dim(\mathfrak{u})$.
			
			If $n>3$, ${\rm GKdim}\:L(\lambda)<\dim(\mathfrak{u})$.
			
	\item When $z=-2$, then
$$
		(\lambda+\rho)^-=(-n+2,\dots,0,-1,1,0,\dots,n-2),
$$
$$
		p((\lambda+\rho)^-)^{ev}=(n-1,1),
$$
	\begin{align*}
		{\rm GKdim}\:L(\lambda)&=n^2-n-(0\cdot (n-1)+ 1)\\&=n^2-n-1=\dim(\mathfrak{u}).
	\end{align*}   	
	So $z=-2$ is an irreducible point.\\	
    \end{enumerate}

\item If $z\in\dfrac{1}{2}+\mathbb{Z}$, we can get that 
\[
{\rm GKdim}\:L(\lambda)=\left\{
\begin{array}{ll}
	n^2-n-(\frac{1}{4}n^2-\frac{1}{2}n)-1 &\text{if $n$ is  even}\\	    	  	
	n^2-n-\frac{1}{4}(n-1)^2 &\text{if $n$ is odd}
\end{array}	
\right.
\]
for $z=-\dfrac{1}{2}$ and  ${\rm GKdim}\:L(\lambda)=n^2-n-1$ for  $z=-\dfrac{3}{2}$. It is not difficult to draw the conclusion.

	\item If $z\notin\mathbb{Z}$ and $z\notin\dfrac{1}{2}+\mathbb{Z}$, $M_{I}(\lambda)$ is irreducible	by (\ref{6}).\\
\end{enumerate}	
 	\item If $p=n$, the process is roughly the same as the case when $p=n-1$, and we omit the process. 		
 \end{enumerate}				
 \end{proof}
  
\section*{Acknowledgements} We would like to thank the anonymous referees for providing many	constructive comments and helping in improving the contents of our paper.  
  
\bibliographystyle{plain}
\bibliography{reference.bib}

\end{document}